\documentclass[AMS]{article} 

\usepackage{xcolor}
\usepackage{geometry} 
\usepackage{graphicx} 

\usepackage{booktabs} 
\usepackage{array} 
\usepackage{paralist} 
\usepackage{verbatim} 
\usepackage{subfig} 

\usepackage{fancyhdr} 
\usepackage{amssymb}
\usepackage{amsmath}
\usepackage{epstopdf}
\usepackage[T1]{fontenc}
\usepackage{t1enc}
\usepackage[utf8]{inputenc}

\usepackage{sectsty}
\allsectionsfont{\sffamily\mdseries\upshape} 

\usepackage{hyperref}

\usepackage[nottoc,notlof,notlot]{tocbibind} 
\usepackage[titles,subfigure]{tocloft} 

\newtheorem{theorem}{Theorem}[section]
\newtheorem{lemma}[theorem]{Lemma}

\newtheorem{fact}[theorem]{Fact}
\newtheorem{cor}[theorem]{Corollary}

\newtheorem{claim}[theorem]{Claim}

\newtheorem{defi}[theorem]{Definition}

\newcommand{\ep}{\varepsilon}

\newcommand{\cB}{\mathcal{B}}
\newcommand{\cD}{\mathcal{D}}

\def\qedf{\hfill $\Box$}

\title{Decomposition of degree-regular graphs into quasi-random pairs without the Regularity lemma} 
\author{ B\'ela Csaba\thanks{The research of the author was supported
by the Ministry of Innovation and Technology of Hungary from the National Research, Development and
Innovation Fund, project no. TKP2021-NVA-09. E-mail: bcsaba@math.u-szeged.hu. ORCID: https://orcid.org/0000-0002-6696-3219}\\ Bolyai Institute, University of Szeged, Hungary. }
\date{} 

\begin{document}
\maketitle

\begin{abstract}
The Szemer\'edi Regularity Lemma, in combination with the Blow-up Lemma, form the Regularity Method, a fundamental tool in graph embeddings, albeit restricted to very large and dense graphs. We propose an alternative vertex-partitioning framework that remains effective even as the density tends to zero and without requiring astronomically large vertex sets. This approach, while narrower in scope, extends regularity-type techniques to relatively small graphs previously inaccessible to the Regularity  Method. As an application,
we use this novel vertex-partitioning method for bipartite packing problems.
\end{abstract}




\section{Introduction}

The Regularity Lemma of Endre Szemer\'edi~\cite{Szemeredi} together with the Blow-up Lemma~\cite{KSSzBl, KSSzBl2} constitutes the so-called Regularity Method. This is an extremely powerful tool, which has found numerous applications in graph embedding problems. However, due to the Regularity Lemma, it is only useful for very large graphs whose edge density is bounded below by some positive constant.

We shall employ an alternative method that in some cases may serve as a substitute for the Regularity Lemma, and which remains applicable even when the density tends to zero (although slowly). Moreover, it does not require that the number of vertices in the graph is extremely large. This alternative method also provides regular pairs, and can therefore be used in conjunction with the Blow-up Lemma.

A result~\cite{Gowers} of Gowers, improved further by Conlon and Fox~\cite{CF},  asserts that the number of clusters in the Regularity Lemma has a tower-type 
lower bound. Consequently, in general one cannot partition the vertex set of a graph into disjoint clusters so that 
almost all edges are in quasirandom cluster pairs. In our approach only a tiny proportion of the edges of the graph
is kept. These edges are arranged into vertex disjoint bipartite subgraphs. 
The union of these subgraphs is a super-matching in which every vertex is a cluster, and every super-edge is a pair of clusters with quasirandom edge distribution between them. Our starting point is an earlier edge decomposition theorem of the author~\cite{aus}. We use this result in conjunction with a randomized algorithm in order to obtain the vertex partition 
having the above properties.    

In our approach the graph to be decomposed cannot be arbitrary: it is crucial that the degrees
are about the same. Fortunately, this is not a strong restriction in many cases, we demonstrate this
via an application for bipartite packing problems in degree-regular or approximately degree-regular\footnote{It is somewhat unfortunate,
that the word {\em regular} in graph theory refers to a quasirandom pair, and also to a graph in which every degree is the same. Here we refer to the second meaning when we write `degree-regular'.} graphs.
Given two graphs, $G$ and $H,$ an {\it $H$-packing} in $G$ is a collection of vertex disjoint copies of $H$
in $G.$
The cases when $\chi(H)\ge 3$ are well-understood, due to the theorems by Komlós, Sárközy and Szemerédi~\cite{KSSz}, and by Kühn and Osthus~\cite{KO09}. When $H$ is bipartite, the problem becomes different, this question was first
considered by Kühn and Osthus~\cite{KO}. A similar problem on packing by subdivision was asked
by Verstra\"ete~\cite{V}. Recently there has been significant new developments in both the bipartite packing~\cite{LMS} and the packing by subdivision~\cite{MPRT} problems. 
Since in both problems the host graph is degree-regular or approximately degree-regular,
this gives us the opportunity to use our vertex decomposition theorem for some cases which
were not considered so far.


The paper consists of two main parts. In the first part we prove a decomposition theorem for (approximately) degree-regular graphs. In this part first we review the basic definitions and tools we need, and then state
and prove the vertex decomposition theorem. After that, in the second part of the paper, we focus on the bipartite $H$-packing and the packing by subdivision problems, first discussing some further tools and
finally proving our results.

 \section{Notation, definitions, main tools}

We only consider simple graphs in this paper. Given a graph $G=(V, E)$ we use the notation $v(G)=|V|$ and $e(G)=|E|.$ Given a set 
$X\subset V,$ $G[X]=(X, E_X),$ where $E_X=\{uv: uv\in E, \ u, v\in X\}.$ For 
disjoint subsets $X, Y\subset V,$ we let 
$G[X, Y]$ denote the bipartite subgraph of 
$G$ with parts $X$ and $Y$ that contains all the edges of $G$ with one endpoint in $X$ and the other endpoint in $Y.$ For every vertex $v\in V,$ the {\it neighborhood} of $v$ is denoted by 
$N_G(v),$ and the {\it degree} of $v$ is denoted by $\deg_G(v)=|N_G(v)|.$ Given a set $S\subset V,$ we let $N_G(v, S)=N_G(v)\cap S$ and 
$\deg_G(v, S)=|N_G(v, S)|,$ the subscripts may be omitted. For an $H\subseteq G$ and $v\in V(G),$ the number of neighbors of $v$ in $H$
is sometimes denoted by $deg_G(v, H).$ If we write $G=(A, B; E),$ this means that $G$ is bipartite with parts $A$ and $B,$ and edge set $E.$ If it is clear from the context, that a graph in question is bipartite, we may only write out
the vertex parts, and omit the letter ``$E$''.

The {\it density} of $G$ is defined to be $d_G=e(G)\cdot {\binom{v(G)}{2}}^{-1}.$ The {\it bipartite density} of bipartite subgraphs of $G$ with parts $A$ and $B$ 
is $d_G(A, B)=\frac{e(G[A, B])}{|A|\cdot |B|}.$ Sometimes the subscript may be omitted when there is no confusion. Similarly, when a graph in question is bipartite, density will mean bipartite density.

Given numbers $x, y$ we say that $z=x\pm y,$ if $x-y \le z \le x+y.$ If $n\ge 1$ is an integer, then we let 
$[n]=\{1, \dots, n\}.$ 
For two numbers, $\alpha$ and  $\beta,$ the notation ``$\alpha \ll \beta$'' means that $\alpha$ is sufficiently smaller than $\beta.$ We remark, that whenever this notation is used in the paper, the relation of $\alpha$ and 
$\beta$ can be explicitly calculated. 
Using the notation ``$\ll$'' enables us to concentrate on the essential parts of the proofs.

\subsection{Regular pairs and bundles}

Let $\cB$ denote the class of {\it balanced} bipartite graphs, that is, bipartite graphs having equal-sized parts, and for a positive integer $m,$ let ${\cB}_m$ denote the class of balanced bipartite graphs having $m$ vertices in both parts. If a bipartite graph has unequal vertex parts, we call it {\it unbalanced}.

\medskip

\begin{defi} 
Let $0<\ep, \delta <1$ be real numbers.
We say that a bipartite graph $H=(A, B)$ is an $\varepsilon$-{\em regular pair}, if for every $X\subseteq A,$ $Y\subseteq B$ 
with $|X|\ge \varepsilon |A|$ and $|Y|\ge \varepsilon |B|$ we have 
$$|d_H(A, B) - d_H(X, Y)|\le \ep.$$ 

We call $H$ an $(\varepsilon, \delta)$-{\em super-regular pair}, if in addition every $v\in A$ has at least $\delta |B|$ neighbors and every $u\in B$ has at least 
$\delta |A|$ neighbors.

We call $H$ an {\em $(\ep, \delta)$-bundle,} if it is an $\ep$-regular pair, and for every vertex $v\in A$ we have 
$deg(v)=\delta |B| \pm \ep |B|,$ and $deg(u)=\delta |A| \pm \ep |A|$ for every $u\in B.$
\end{defi}

Note, that an $(\ep, \delta)$-bundle is also an $(\varepsilon, \delta-\ep)$-super-regular pair,
and has density $\delta\pm \ep.$

The following well known fact below will prove to be useful, the proof is omitted.

\begin{fact}\label{reg}
Assume that $H=(A, B)$ is an $\ep$-regular pair with density $d_H.$ Let $A'=\{x\in A: deg(x)<(d_H-\ep)|B|\}$ and 
$A''=\{x\in A: deg(x)>(d_H+\ep)|B|\}.$ Similarly, let $B'=\{x\in B: deg(x)<(d_H-\ep)|A|\}$ and 
$B''=\{x\in B: deg(x)>(d_H+\ep)|A|\}.$ Then $|A'|, |A''|\le \ep |A|$ and $|B'|, |B''|\le \ep |B|.$  
\end{fact}

We will use the so called Slicing lemma~\cite{KS} at various points in the paper.

\begin{fact}\label{slicing}
Assume that $H=(A, B)$ is an $\ep$-regular pair with density $d_H,$ and for some $\alpha >\ep$ let $A'\subset A,$
$|A'|\ge \alpha |A|$ and  $B'\subset B,$ $|B'|\ge \alpha |B|.$ Then $(A', B')$ is an $\ep'$-regular pair with 
$\ep'=\max \{\ep/\alpha, 2\ep\}$ and for its density $d'$ we have $|d'-d|<\ep.$ 
\end{fact}

We are going to need the lemma below.

\begin{lemma}\label{bundle}
Let $0<\ep<1/10$ and $3\ep \le d\le 1/2,$ and assume that $F\in \cB_m$ is an $\ep$-regular pair with density $d.$ 
Then there is a $(3\ep, d)$-bundle $H\subset F,$ $H\in \cB_{m_1}$ with density $d\pm \ep$ 
such that $m_1\ge (1-2\ep)m.$ Moreover, $e(H)\ge e(F)-2\ep m^2-2\ep (d-\ep)m^2\ge e(F)-4\ep m^2.$
\end{lemma}

\noindent {\bf Proof:} 
Discard the vertices that have less than $(d-\ep)m$ neighbors, and then those, which have more than $(d+\ep)m$ neighbors. By Fact~\ref{reg} we discarded at most 
$2\ep m$ vertices from both parts of $F.$ Then, if necessary, we discard some further vertices from the larger part to make the subgraph balanced. This way we have obtained $H,$ which has $m_1\ge (1-2\ep)m$ vertices in both parts. 
Clearly, if we take arbitrary subsets of sizes at least $2\ep m_1\ge \ep m$ from both parts, the density of the subgraph spanned by
these subsets is $d\pm \ep$ by $\ep$-regularity of $F.$ Hence, the density of $H$ is $d\pm \ep.$ 

Since every remaining vertex that belongs to $H$ lost at most $2\ep m$ neighbors, this shows that
$$(d-3\ep)m_1\le (d-3\ep)m\le deg(v)\le (d+\ep)m\le (d+3\ep)m_1$$ for every vertex $v\in V(H),$ here the last inequality follows from the upper bounds for $\ep$ and $d.$ Finally, we get the claimed lower bound for $e(H)$ since from $e(F)$ we subtracted
at most $(d-\ep)m$ edges for every discarded vertex that had low degree, and at most $m$ edges for every high degree discarded vertex.
\qedf

\medskip

\begin{cor}\label{kov}
Let $0<\ep < 1/10$ and $2\sqrt{\ep}\le d\le 1/2. $ If $F\in \cB_m$ is an $\ep$-regular pair with density 
$d,$ then it contains a $(3\ep, d)$-bundle $H\subset F,$ $H\in \cB_{m_1}$ with density $d\pm \ep$ 
such that $m_1\ge (1-2\ep)m,$ and $e(H)\ge (1-d)e(F).$
\end{cor}

\noindent {\bf Proof:} This follows easily from the condition that $4\ep \le d^2.$\qedf

\subsection{A probabilistic inequality}

We use random methods at various points in the paper, and need large deviation bounds for discrete probability distributions.
The following inequality, a generalized version of Chernoff's bound,  is Theorem 2.8 in~\cite{JLR}.

\begin{theorem}\label{Chernoff} 
Assume that $X$ is the sum of $k$ independent indicator random variables: $X=X_1+\ldots+X_k.$ If $0\le \mu \le 3/2,$ then
$$P(|X-\mathbb{E}[X]|\ge \mu \mathbb{E}[X])\le 2e^{-\frac{\mu^2}{3} \mathbb{E}[X]}.$$ 
\end{theorem}

\subsection{Criterion for quasi-randomness}
Working with random subgraphs of a regular pair spanned by random subsets is crucial 
in the paper. 
For proving quasi-randomness of a random subgraph we obtain this way, we use the following result by Kohayakawa and Rödl from~\cite{KohRdl}:

\begin{theorem}\label{KohRod}
Let $\eta$ be a constant with $0<\eta <1.$ Let $G = (V, E)$ be a graph with
$(A, B)$ a pair of disjoint, nonempty subsets of $V$ with $|A| \ge 2/\eta.$  Set 
$\rho= d(A, B) = e(A, B)/(|A|\cdot |B|).$ Let $\cD$ be the collection of all pairs $\{x, x'\}$ of vertices of $A$ for which
\begin{itemize}
\item [(i)] $deg(x, B), deg(x', B) \ge (\rho-\eta)|B|,$
\item [(ii)] $deg(x, N(x')\cap B) \le (\rho+\eta)^2|B|.$
\end{itemize}
Then if $|\cD| > (1-5\eta)|A|^2/2,$ the pair $G[A, B]$ is $(16\eta)^{1/5}$-regular.
\end{theorem}

\section{The vertex decomposition theorem}

The edge decomposition theorem below by the author~\cite{aus} is based upon a result  by Peng, R\"odl and Ruci\'nski~\cite{Peng}. 

\begin{theorem}\label{dekomp-epsz}
Let $G=(V, E)\in \cB_n$ with density $d_G,$ and let $0< d < 1$ and $0<\ep<1/2$ such that $n> \exp(50 \log (1/d) /\ep^2).$ 
Then there exist natural numbers $m=m(\ep, d)$ and $K=K(\varepsilon, d)$ such that the following holds: 
$E(G)$ can be written as the edge-disjoint union of bipartite graphs $H_1, \ldots, H_{K}\in \cB,$ and another, {\em exceptional} graph $H_0,$ such 
that for every $i\ge 1,$ $H_i$ is an $\varepsilon$-regular pair with $m_i\ge m$ vertices in both parts and density $d_{i}\ge (1-\ep/3)d,$
while $H_0$ has density less than $d.$ Furthermore, $$m \ge \frac{1}{2}d^{50/\varepsilon^2}n$$ and 
$$K \le 8 \frac{d_G}{d}\cdot d^{-100/\varepsilon^2}.$$
\end{theorem}
\medskip

Observe, that if $d>d_G,$ then the theorem holds with $K=0.$ The theorem becomes useful when 
$d_G\ge d.$ It is also clear, that one needs $m\gg 1/\ep,$ which in turn implies $\ep =\Omega(\sqrt{\log (1/d)/\log n}).$ 

In Theorem~\ref{dekomp-epsz} the edge set is partitioned into $K=K(\ep, d)$ regular pairs, where $K(\ep, d)$ is a single exponential function
of $\ep$ and $d.$ Since the vertex parts of different regular pairs may intersect, this reduces the applicability of the above theorem. 
Still, one can use it under relatively general circumstances for certain kind of vertex decompositions. 
Our goal is to show that for graphs in which the vertex degrees
are linear (or slightly sublinear) in $n$ and the graph is almost degree-regular, Theorem~\ref{dekomp-epsz} can be applied for finding a {\em vertex partition}, which, although
one has to discard the majority of edges, proves to be useful for embedding problems. 

The main decomposition theorem of the paper is as follows.

\begin{theorem}\label{pontfelbontas}
Let $0< \ep <1/10$ be a number, and assume that $d$ is a real number such that $10 \ep ^{1/5}\le d< 1/3.$  
Assume further that $G=(V, E)\in \cB_n$ with bipartition $V=A\cup B$ and $r(1-\lambda)\le deg(v)\le r$ 
for every $v\in V,$ where $r\ge d^{1/3}n$ and $0\le \lambda\le d^4/2.$ 
If $n>\exp(150 \log (1/d)/\ep^2),$ then there exists a natural number $K\le 8 \frac{d_G}{d}\cdot d^{-100/\varepsilon^2}$ such that $G$ admits the following decomposition:

\begin{itemize}

\item [(i)] $A=A_1\cup \dots \cup A_K \cup A_0$ and $B=B_1\cup \dots \cup B_K \cup B_0,$
where $A_i\cap A_j=B_i\cap B_j=\emptyset,$ whenever $i\neq j,$

\item [(ii)] $|A_0 \cup B _0|< 8\sqrt[3]{d}n,$ 

\item [(iii)] $|A_i|, |B_i| \ge d^{101/\ep^2}\cdot \frac{n}{4}=\Omega\left(n^{1/3}(\log n)^{-1/2}\right)$ for every 
$1\le i\le K,$

\item [(iv)] $||A_i|-|B_i||\le 2\ep^2|A_i|$ for every $1\le i\le K,$

\item [(v)] the bipartite subgraphs $G[A_i, B_i]\subset G$ for $i\ge 1$ are all $(\ep', d_i)$-bundles, where 
$\ep'\le (64\ep)^{1/5}$ and $d_i\ge d-2\ep.$ 

\end{itemize}

\end{theorem}

Before turning to the proof of the theorem, let us discuss some details and features of it.
The sets $A_1, \ldots, A_K$ and $B_1, \ldots, B_K$ are called the {\it non-exceptional clusters} of the decomposition,
and the sets $A_0$ and $B_0$ are the {\it exceptional clusters,} analogously to the decomposition of the Regularity lemma. 
Note, that the exceptional clusters, while can be made small by choosing $d$ to be small, could be much larger than the
non-exceptional clusters. Analogously to the Regularity lemma, sometimes we will call the union of the
$G[A_i, B_i]$ $(i\ge 1)$ bundles the {\it reduced graph} of $G.$ The vertices of this reduced graph are the clusters, the edges are the 
$G[A_i, B_i]$-bundles, these form a matching in the reduced graph.

We will use Theorem~\ref{dekomp-epsz} as a black box when proving Theorem~\ref{pontfelbontas}.  Note, that the bound for $n$ as a function of $\ep$ and $d$ is not the same in the two theorems,
in the second one the constant multiplier in the exponent is three times larger. We need it for conditions $(iii)$ and $(iv).$
Having said this, we also remark that the we did not optimize on the constants in the theorem. 

Theorem~\ref{dekomp-epsz} has an algorithmic version, Theorem 5.1 in~\cite{aus}. It allows one to formulate an algorithmic version of Theorem~\ref{pontfelbontas}, which can be proved similarly to Theorem~\ref{dekomp-epsz}.
However, $\ep$ is raised to a larger power in it, and here we do not focus on algorithmic applications, so we decided to work with Theorem~\ref{dekomp-epsz}. 

One can express a lower bound for $\ep$ and $d$ as functions of $n.$ Easy calculation, which we leave for the reader, shows the following fact.

\begin{fact}\label{epszd} 
We have $$\ep\ge\left(\frac{\log\log n}{\log n}\right)^{1/2} \quad \quad {\rm and \ } \quad \quad d\ge 10\left(\frac{\log\log n}{\log n}\right)^{1/10}.$$
 \end{fact} 

\medskip

Note that while the Regularity lemma gives a decomposition of the vertex set of a graph into clusters such that the bipartite subgraphs between {\it almost all} pairs of clusters are quasirandom, and, crucially, we still have {\it almost all} edges,
the above theorem gives a much weaker decomposition. In particular, the $G[A_i, B_i]$ bundles may have only a very small fraction of the edges of $G$ itself. 
Due to this fact, one would need extra work for providing connections between clusters that do not 
belong to the same ``super-edge'' when using Theorem~\ref{pontfelbontas} for embedding a large
connected subgraph (see in~\cite{Cs}), but this problem does not emerge in
the applications we consider in the second part of the paper.
 




\subsection{Proof of Theorem~\ref{pontfelbontas}}\label{biz}

The proof consists of a randomized algorithm and its analysis. We begin with applying Theorem~\ref{dekomp-epsz} for $G$ with parameters $\ep$ and $d \  (\ge 10\ep^{1/5}),$ assuming that 
$n>\exp(150 \log (1/d)/\ep^2).$ We denote by $H_1, \dots, H_K$ the $\ep$-regular pairs
in the decomposition we obtain this way, and $H_0$ denotes the {\it exceptional} bipartite subgraph
having density less than $d.$ Since $m \ge \frac{1}{2}d^{50/\varepsilon^2}n,$ using the lower 
bound for $n$ we have that $m\ge n^{2/3}/2.$ 

\smallskip

Next we transform every regular pair into a bundle: we apply Lemma~\ref{bundle} 
for every $H_i,$ $i\ge 1,$ and for simplicity we keep the notation for the $H_i$ bundles. By Corollary~\ref{kov} we do not increase the number of edges of $H_0$ by more than 
$d \cdot e(G) \le dn^2$ edges, so the new $H_0$ has density at most $2d.$

After this we assign a probability $p(H_i)$ to every non-exceptional bundle $H_i$ $(i\ge 1)$:
$$p(H_i)=\frac{(d(H_i)-\ep)m_i}{r}.$$ Observe, that if $v\in V(H_i),$ then $deg(v, H_i)/m_i\in d(H_i)\pm \ep,$ hence, if we randomly choose an edge incident to  $v,$ then the probability that this edge belongs
to $H_i,$ is about $p(H_i) $ 

For every $v\in V(G)$ we define a set: let $S_v=\{i: i\ge 1, v\in V(H_i)\}.$ Next we estimate the sum of the $p(H_i)$
probabilities for $i\in S_v.$ 

\begin{lemma}
For every $v\in V(G)$ we have $$1-\frac{deg(v, H_0)}{r}-d^4 \le \sum_{i: i\in S_v}p(H_i) \le 1-\frac{deg(v, H_0)}{r}.$$
\end{lemma}

\noindent {\bf Proof:} We begin with the upper bound:
$$\sum_{i: i\in S_v}p(H_i)=\sum_{i: i\in S_v} \frac{(d(H_i)-\ep) m_i}{r}\le \sum_{i: i\in S_v} \frac{deg(v, H_i)}{r}\le
\frac{r-deg(v, H_0)}{r}=1-\frac{deg(v, H_0)}{r}.$$ 

Next we prove the lower bound. Recall, that $0\le \lambda \le d^4/2,$ and every vertex has degree in the
interval $[r(1-\lambda), r].$ 
Using that $(d(H_i)+\ep)m_i$ is an upper bound for the degrees of vertices in the bundle $H_i$ 
and $deg(v)\ge r-d^4r/2,$ we have 
$$\sum_{i: i\in S_v}(d(H_i)+\ep)m_i \ge r-\frac{d^4r}{2}-deg(v, H_0),$$ which implies that
$$\sum_{i: i\in S_v} (p(H_i)+\frac{2\ep m_i}{r})\ge 1- \frac{d^4}{2}- \frac{deg(v, H_0)}{r}.$$ 

Since $\ep \le d^5/10,$ we have
$$\sum_{i: i\in S_v}2\ep m_i\le \sum_{i: i\in S_v} d^5m_i/5 \le d^4\sum_{i: i\in S_v} d(H_i)m_i/5\le \frac{d^4r}{4}.$$
This implies that $$1-\frac{deg(v, H_0)}{r}-d^4\le \sum_{i: i\in S_v} p(H_i) \le 1-\frac{deg(v, H_0)}{r},$$
proving what is desired. \qedf

\medskip

Next for every $v\in V(G)$ we define $$p_0(v)=1-\sum_{i: i\in S_v}p(H_i).$$ 
The above lemma easily implies the following.

\smallskip

\begin{cor} For every $v\in V(G)$ we have
that $$\frac{deg(v, H_0)}{r}\le p_0(v)\le \frac{deg(v, H_0)}{r}+d^4.$$
\end{cor}
\medskip

Given a number $\psi \in (0,1)$, let $L=\{v\in V(G): deg(v, H_0)\ge \psi n\}.$ 

\begin{fact}\label{Lmeret} 
We have $$|L|\le \frac{2e(H_0)}{\psi n} \le \frac{4dn^2}{\psi n}= \frac{4dn}{\psi}.$$ 
\end{fact}

In order to obtain the partitions of $A$ and $B$ as stated in the theorem, we use a randomized algorithm. Given any vertex $v\in V(G),$ we ``roll
a dice'' which has $|S_v|+1$ faces, the numbers written on the faces are the elements of $S_v$ and $0.$
The probability of the outcomes are as follows:
$$P({\rm the \ outcome \ is \ } i\in S_v) =p(H_i),$$ and $$P({\rm the \ outcome \ is \ } 0) =p_0(v).$$
Note, that the number of faces and the probabilities of the outcomes may differ for different vertices.

\medskip

Next for every vertex we roll its dice, independently from other vertices. Say, that vertex $v\in A.$ If the outcome
is some $j\in S_v,$ then we put $v$ into $A_j.$ Otherwise, if the outcome is 0, $v$ will belong to $A_0.$ Analogous procedure is applied to the vertices in $B.$

After finishing this random distribution procedure for every vertex, we obtain the subgraphs 
$G[A_i, B_i]$ for every $1\le i\le K.$ Observe, that every vertex in $V(H_i)$ had the same
probability $p(H_i)$ for belonging to $A_i\cup B_i,$ that is, $G[A_i, B_i]$ is a random induced subgraph of $H_i.$ 
Moreover, $p(H_i)> dm/(2r) \gg n^{-1/3}(\log n)^{-1/2}$ for $i=1, \dots, K,$ using Fact~\ref{epszd}, since $m_i\ge n^{2/3}/2$ and $r\le n.$ We 
need the following lemma.

\begin{lemma}\label{velresz}
Assume that $F[X, Y]$ is an $(\ep, d)$-bundle with $|X|=|Y|=k,$ where $k$ is a sufficiently large integer, and $\ep, d$ are two real numbers such that $0<d<1/3$ and 
$1/\sqrt{\log k}<\ep \le d/10.$ 
Let $p$ be a real number
with $(30 \log^3 k)/k \le p\le 1.$ 
Assume that $X_R$ is a random subset of $X,$ such that every $x\in X$ belongs to $X_R$
with probability $p,$ independently from other vertices. We obtain $Y_R\subseteq Y$ analogously. 
Then with probability at least $1-1/k^{3},$ the induced subgraph $F[X_R, Y_R]\subset F[X, Y]$ is a
$((64\ep)^{1/5}, d-2\ep)$-bundle with density $\rho \in [d-2\ep, d+2\ep],$
moreover, $(1-\ep^2)pk\le |X_R|, |Y_R| \le (1+\ep^2)pk.$ 
\end{lemma}

\noindent {\bf Proof:} Throughout the proof of the lemma we say that an event holds with high probability,
if it holds with probability at least $1-1/k^{5}.$ 

For estimating the cardinalities of $X_R$ and $Y_R$ we use Theorem~\ref{Chernoff} (the Chernoff bound) with parameter 
$\mu=\ep^2\ge 1/\log k.$
Note, that $\mathbb{E}|X_R|= \mathbb{E}|Y_R|=pk\ge 30 \log^3 k.$ Then we have 
$$P(||X_R|-pk|\ge \mu pk)\le 2e^{-\mu^2pk/3}\le 2e^{-10 \log k}<k^{-5}.$$ Similar bound holds for $|Y_R|.$

In order to prove that $F[X_R, Y_R]$ is $(64\ep)^{1/5}$-regular, we will check the conditions of Theorem~\ref{KohRod}.
We begin with determining high probability upper and lower bounds for the
density $\rho$ of $F[X_R, Y_R].$ Clearly, $$\frac{\min\{deg(x, Y_R): x\in X_R\}}{|Y_R|} \le \rho=\frac{\sum_{x\in X_R} deg(x, Y_R)}{|X_R|\cdot |Y_R|}\le \frac{\max\{deg(x, Y_R): x\in X_R\}}{|Y_R|}.$$ 

Since $F[X, Y]$ is an $(\ep, d)$-bundle, every vertex in $X$ has degree in the interval $[(d-\ep)|Y|, (d+\ep)|Y|].$ 
Applying the Chernoff bound we have $$(1-\ep)(d-\ep)p|Y| \le deg(x, Y_R) \le (1+\ep)(d+\ep)p|Y|$$ 
and $$(1-\ep)(d-\ep)p|X| \le deg(y, X_R) \le (1+\ep)(d+\ep)p|X|$$for every $x\in X$ and $y\in Y$ with
high probability. Hence, we have $$\frac{\min\{deg(x, Y_R): x\in X_R\}}{|Y_R|} \ge \frac{(d-\ep)(1-\ep)p|Y|}{(1+\ep)p|Y|} =
(d-\ep)\left(1-\frac{2\ep}{1+\ep}\right)\ge d-2\ep,$$
and $$\frac{\max\{deg(x, Y_R): x\in X_R\}}{|Y_R|} \le \frac{(d+\ep)(1+\ep)p|Y|}{(1-\ep)p|Y|} =(d+\ep)\left(1+\frac{2\ep}{1-\ep}\right)\le d+2\ep.$$

That is, $(d-2\ep)|Y_R|\le deg(x, Y_R)\le (d+2\ep)|Y_R|$ for every $x\in X,$ which implies that $d-2\ep\le \rho \le d+2\ep.$ Analogously, we have that $(d-2\ep)|X_R|\le deg(y, X_R)\le (d+2\ep)|X_R|$ for every $y\in Y$
with high probability.

Using the above lower and upper bounds for $\rho$ we can verify condition $(i)$ of Theorem~\ref{KohRod}: every $x\in X_R$ has degree at least $(d-2\ep)|Y_R|\ge (\rho-4\ep)|Y_R|$ with high probability.

For condition $(ii)$ of Theorem~\ref{KohRod} we need upper bounds for the co-degrees of the vast majority of vertex pairs in $X_R.$
We define a subset of $X$ for every $x\in X$: let $\widehat{X}(x)=\{x'\in X: deg(x, N(x'))> (d+\ep)^2|Y|\}.$
Since $F[X, Y]$ is an $(\ep, d)$-bundle, $|\widehat{X}(x)|\le \ep|X|$ for every $x\in X.$ By the Chernoff bound, 
$|X_R\cap \widehat{X}(x)|\le 2\ep |X_R|$ for every $x\in X$ with high probability.

Assume now, that $x, x' \in X_R$ such that $x'\in X_R-\widehat{X}(x).$ Then the expected number of neighbors
of $x$ in $N(x')\cap Y_R$ is at most $(d+\ep)^2p|Y|.$
The high probability lower bound $|Y_R|\ge (1-\ep^2)p|Y|$ implies that $\frac{1+\ep}{1-\ep^2}|Y_R|\ge (1+\ep)p|Y|.$
Since $\ep\le d/10$ and $d<1/3,$ we get that $(d+\ep)^2/(1-\ep^2)\le (d+2\ep)^2.$ Using the Chernoff bound we
have 
$$deg(x, N(x')\cap Y_R)\le (d+\ep)^2(1+\ep^2)p|Y|\le (d+2\ep)^2|Y_R|\le (\rho+4\ep)^2|Y_R|$$ with high probability.

Throughout the proof we have to control the probabilities of less than $k^2$ events, each having probability at least $1-1/k^{5}$: 
\begin{itemize}
\item the cardinalities of $X_R$ and $Y_R$ are close to their expectation,

\item every vertex in $X\cup Y$ has about the expected number of neighbors in $Y_R$ and $X_R,$ respectively, 

\item the number of common neighbors of the $\{x, x'\}$ pairs , where $x'\in X_R-\widehat{X}(x),$ are also close
to their expected values,

\item for every $x\in X$ the cardinality of $X_R\cap \widehat{X}(x)$ is close to its expectation. 
\end{itemize}

Hence, the intersection of these events has probability at least $1-k^2/k^{5}=1-1/k^3.$
Therefore, Theorem~\ref{KohRod} implies that $F[X_R, Y_R]$ is a $(64\ep)^{1/5}$-regular pair. Recall, that
$(d-2\ep)|Y_R|\le deg(x, Y_R)\le (d+2\ep)|Y_R|$ for every $x\in X$ and $(d-2\ep)|X_R|\le deg(y, X_R)\le (d+2\ep)|X_R|$ for every $y\in Y.$ Since $(64\ep)^{1/5}>4\ep,$ $F[X_R, Y_R]$ is a $((64\ep)^{1/5}, d-2\ep)$-bundle. 
\qedf

\medskip

We apply Lemma~\ref{velresz} for every bundle $H_1, \dots, H_K.$ Recall, that the size of the parts of
any $H_i$ are at least $n^{2/3}/2,$ this number plays the role of $k$ in Lemma~\ref{velresz}. Hence, the probability 
that a random subgraph of some $H_i$ bundle fails to be $((64\ep)^{1/5}, d-2\ep)$-bundle, is at most $2/n^2.$ There
are $K<n$ bundles, therefore, with probability at least $1-2/n$ the random subgraphs $G[A_1, B_1], \dots, G[A_K, B_K]$ are all $((64\ep)^{1/5}, d-2\ep)$-bundles. The bounds $|A_i|, |B_i| = \Omega(n^{1/3}/\sqrt{\log n})$ are also implied by Lemma~\ref{velresz},
since $m\ge n^{2/3}/2$ and $p\ge n^{-1/3}(\log n)^{-1/2}.$

\medskip

Finally, we estimate the cardinality of the set $A_0\cup B_0$ from above.  Recall the definition of
the set $L$: for a given parameter $\psi\in (0, 1)$ this set contains those vertices that have at least $\psi n$ edges incident to them from $H_0.$  Assume first, that $v\in V(G)-L.$ Then $p_0(v)< d^4+\psi n/r.$ Even in the worst
case, when all vertices of $L$ belong to $A_0\cup B_0,$ we have 
$$\mathbb{E}|A_0\cup B_0| < (d^4+\psi n/r)n +|L|\le d^4n+\frac{\psi n^2}{r}+\frac{4dn}{\psi},$$
here we used Fact~\ref{Lmeret}.
By Theorem~\ref{Chernoff}, we have that with high probability 
$$|A_0\cup B_0|\le 2d^4n+\frac{2\psi n^2}{r}+\frac{5dn}{\psi}.$$

Set $\psi=d^{2/3}.$ Using that $r\ge d^{1/3}n,$ we get that the following upper bound holds
with high probability:
$$|A_0\cup B_0|\le 2d^4n+2d^{1/3}n+5d^{1/3}n< 8d^{1/3}n.$$

With this we have finished the proof of Theorem~\ref{pontfelbontas}. \qedf



\section{Packing of bipartite graphs}

Given two graphs, $G$ and $H,$ an {\it $H$-packing} in $G$ is a collection of vertex disjoint copies of $H$
in $G.$
Clearly, the more vertex disjoint copies of $H$ we want to find in $G,$ the harder the task. 
If $G$ admits an $H$-packing which covers every vertex, we call it a {\it perfect $H$-packing} or 
an {\it $H$-factor.}

A special case of $H$-packings is the well-understood case of matchings, that is, when $H$ is a single edge. When $H$ is 
a clique on at least three vertices, the problem becomes considerably harder. If $G$ is an $n$-vertex graph, $r\ge 2$ is an
integer, $r$ divides $n,$ and $\delta(G)\ge (1-1/r)n,$ then $G$ has a perfect $K_r$-packing (the case $r=3$
is proved by Hajnal and Corr\'adi~\cite{CH}, the case of $r\ge 3$ is the Hajnal-Szemer\'edi theorem~\cite{HSz}).

Alon and Yuster~\cite{AY}, using the Regularity lemma, generalized the Hajnal-Szemerédi theorem for almost
perfect packings by a fixed $H,$ leaving out $o(n)$ vertices. Their theorem was improved by Koml\'os, S\'ark\"ozy and Szemer\'edi: they proved that for every $H$ there exists a constant $C_H$ such that if $G$ is an $n$-vertex graph, $n$ is divisible 
by $|V(H)|$ and the minimum degree of $G$ is at least $(1-1/\chi(H))n+C_H,$ then $G$ has a perfect $H$-packing. 
For the case when $\chi(H)\ge 3,$ 
K\"uhn and Osthus~\cite{KO09} refined the above result, and proved tight
bounds, depending on the chromatic number and the critical chromatic number of $H.$

It turns out that the case $\chi(H)=2$ is different, it received special attention. 
In~\cite{KO} K\"uhn and Osthus proved the theorem below
on packing with arbitrary bipartite graphs. 

For stating their theorem we need to introduce a notation: given numbers $a \ge b,$ we say that a 
graph $G$ is 
$(a \pm b)$-regular if its minimum degree is at least $a - b$ and its maximum degree is at most $a + b.$
The theorem below is Theorem 1.1. in~\cite{KO}. 

\begin{theorem}\label{KOthm} 
Given a bipartite graph $H$ and constants $0< c, \alpha \le 1,$ there exist positive
numbers $\gamma= \gamma(c, \alpha)$ and $n_0 = n_0(H, c, \alpha)$ such that every $(cn \pm \gamma n)$-regular graph $G$ of order $n \ge n_0$ has an $H$-packing which covers all but at most $\alpha n$ vertices of $G.$
 \end{theorem}

They also conjectured that if $G$ is $cn$-regular (so there is no ``slackness'' in the degrees), then there is an $H$-packing leaving at most a constant 
number of vertices in $G$ uncovered. This question was open for nearly twenty years, until recently, when Letzter, Methuku and Sudakov~\cite{LMS} proved the conjecture.

\begin{theorem}\label{LMSthm}
For every bipartite graph $H$ and every $0 < c \le 1,$ there are constants $C = C(H, c)$ and $n_0=n_0(H, c)$ such that every
$r$-regular graph $G$ of order $n\ge n_0,$ where $r\ge c n,$ has an $H$-packing that covers all but at 
most $C$ vertices of $G.$
\end{theorem}

The proofs of Theorem~\ref{KOthm} and Theorem~\ref{LMSthm} use the Regularity lemma. Hence, the threshold numbers $n_0$ and the constant $C$ are of tower type. Moreover, these
proofs work only if $H$ has bounded degree. The proof of Theorem~\ref{LMSthm} is quite involved, and
in addition to the Regularity lemma it also uses other recent powerful techniques on expanders.

It turns out, that if $H$ is unbalanced, 
then one can find such an $H$-packing relatively easily, which leaves out only a constant number of vertices,
even if the degrees in $G$ are just approximately the same.
The following result was proved by Kühn and Osthus~\cite{KO}.

\begin{theorem}\label{KOthm2} Given a bipartite graph $H$ whose vertex classes have different size and a
constant $0 < c \le 1,$ there exist $\gamma=\gamma(H, C)> 0$ and $C = C(H, c)$ such that every 
$(cn \pm \gamma n)$-regular graph $G$ has an $H$-packing which covers all but at most $C$ vertices of $G.$
\end{theorem}

Since the above theorem, similarly to the previous ones, was proved using the Regularity lemma, the constant $C$ has a tower-type dependence on $H$ and $c.$ 

Our first result on bipartite $H$-packing is the following.

\begin{theorem}\label{tetel1}
Let $\ep, d$ be real numbers such that $0< \ep <1/10$ and $20 \ep ^{1/5}\le d< 1/3.$  
Assume that $G=(V, E)$ is an $(1\pm \frac{d^4}{10})\rho$-regular graph 
on $n$ vertices with $\rho\ge 3d^{1/3}n.$ Assume further that $H$ is a bipartite graph on $h$ vertices with 
maximum degree $D\ge 1.$ 
If 
\begin{itemize}
\item $n>2\exp(150 \log (1/d)/\ep^2)$ and 
\item $h\le \frac{\ep^{1/5}}{16D} \exp(-(D+101/\ep^2) \log (1/d)) 3^{-D} n,$ 
\end{itemize}
then $G$ has an $H$-packing, which covers all but at most $5d^{1/3}n$ vertices of $G.$
\end{theorem}

We emphasize, that in Theorem~\ref{tetel1} the graph $H$ may have unbounded size and degree, and $G$ is allowed to have vanishing density. By Fact~\ref{epszd}, one may choose $\rho=o(n)$ if $\rho\ge 3d^{1/3}n$   
and $d=\Omega((\log \log n/\log n)^{1/30}).$ Besides, $n$ can be single-exponential in a polynomial of 
$1/\ep,$
unlike in the Regularity lemma. Hence, Theorem~\ref{tetel1} extends Theorem~\ref{KOthm} in several ways,
and even Theorem~\ref{LMSthm} in that it works for graphs of ``practical size'', and the size of $H$ 
may grow with $n.$ 

\medskip

We also consider the bipartite packing problem of unbalanced bipartite graphs, and the closely related
packing by subdivison problem. Given a graph $H,$ we obtain a {\it subdivision} of it by replacing edges
of $H$ with internally vertex-disjoint paths. The resulting graph is denoted by $TH,$ as $H$ is a topological
minor of it. 

We focus on the special case of the 1-subdivision, when {\it every edge} of $H$ is replaced by internally
vertex-disjoint paths of length 2. It is easy to see, that the 1-subdivision of any graph is bipartite, since 1-subdivisions cannot contain odd cycles.

In 2002 Verstr\"aete~\cite{V} conjectured that for every graph $H$ and $\eta>0$ there exists a threshold number
$r_0$ such that if $G$ is an $n$-vertex, $r$-regular graph with $r\ge r_0,$ then it contains a $TH$-packing
that leaves at most $\eta n$ vertices uncovered in $G.$ 
Recently this conjecture was proved by Montgomery, Petrova, Ranganathan and Tan~\cite{MPRT}. Note, that $r_0$ does not depend on $n,$ so it is not surprising that the $\eta n$ ``error'' is unavoidable. 
Theorem~\ref{KOthm2} for packing with unbalanced bipartite graphs, as observed by Kühn and Osthus, implies a stronger bound for packing by subdivisions, only leaving out a (very large) constant number of vertices. However, this works only for dense graphs, and the mentioned constant has a tower type bound.

Our theorem for packing with an unbalanced bipartite graph is as follows.

\begin{theorem}\label{tetel2}
Let $H$ be a bipartite graph with vertex parts $X$ and $Y$ such that $|X|>|Y|.$ Set $h=|X|+|Y|.$ 
Assume that $\ep, \rho$ and $d$ are constants and $n$ is
number such that following are satisfied:  
\begin{itemize}
\item $0< \ep <1/10,$ 
\item $10 \ep ^{1/5}\le d< 1/100,$ 
\item $\rho\ge 10 d^{1/9}n.$
\item $n>\exp(150 \log (1/d)/\ep^2).$  

\end{itemize}

Assume further that $G=(V, E)$ is an $(1\pm \frac{d^4}{10})\rho$-regular graph 
on $n$ vertices. Then there exists a number 
$C=C(H, \ep, d)\le \frac{8h}{d}\cdot d^{-100/\varepsilon^2}$ such that $G$ has an $H$-packing, leaving out at most $C$ vertices.
\end{theorem}

\medskip

 

Following the observation by Kühn and Osthus~\cite{KO}, Theorem~\ref{tetel2} implies the corollary below.

\begin{cor}
Given a graph $H$ without isolated vertices which is not a union of cycles and
a constant $0 < c\le 1,$ there exist $\gamma=\gamma(H, c) >0$ and $C = C(H, c)$ such that every 
$(cn \pm \gamma n)$-regular graph $G$ has a packing with 1-subdivisions of $H$ which covers all but at most $C$ vertices of $G.$
\end{cor}

\noindent {\bf Proof:} The statement follows from the fact that if $H$ is not a union of cycles,
then it has more edges than vertices, hence, the 1-subdivision of $H$ must be unbalanced. 
Let $\ep, d$ and $n$ be numbers that satisfy the conditions in Theorem~\ref{tetel2}, and set $\rho=cn,$ 
$\gamma=d^4c/10.$ Applying Theorem~\ref{tetel2} finishes the proof.   
\qedf
  
\medskip
  
Before turning to the proofs of Theorem~\ref{tetel1} and~\ref{tetel2}, let us have a further remark on the size of $H$ in Theorem~\ref{tetel1}. Denote $R(H)$ the Ramsey number of
$H,$ that is, the least integer $n$ such that in any 2-coloring of the edges of $K_n,$ some
monochromatic copy of $H$ must always be formed. Graham, R\"odl and Ruci\'nski~\cite{GRR}
proved that for every $D\ge 1$ and $h\ge 6$ there exists a bipartite graph $H$ on $h$ vertices with maximum degree at most
$D$ which satisfies $R(H)> 2^{cD}h$ for a positive constant $c.$ 

Now set $N=R(H)-1,$ and consider any 2-coloring of the edges of $K_N.$ 
Denote $F$ the subgraph containing the edges of the majority color, then $F$ has $N$ vertices and density $d_F\ge 1/2.$ 
By Theorem~\ref{dekomp-epsz}, $F$ contains an $\ep$-regular pair with parts having cardinality at least $n/2$ where
$n=2^{-50/\ep^2}N.$ Let $G$ denote this $\ep$-regular pair. 
The bound on $R(H)$ implies that $H\not\subset G,$ even though $h\le 2^{-cD}N=2^{50/\ep^2-cD}n,$ while by Theorem~\ref{tetel1} if
$\widetilde{h}\le C_{\ep}\exp(-(D+101\log (1/d)/\ep^2))n/D,$ then $G$ admits an almost perfect $\widetilde{H}$-packing
for some bipartite graph $\widetilde{H}.$ 
Hence, the graphs by which we can pack $G$ are not significantly smaller than the ones which can be found in $G$ at all -- for that we choose $\ep, d$ and $r$ to be constants and the graph $H$ with $D=\Delta(H)$ so that $cD\gg 50/\ep^2.$

\section{The proofs of the packing theorems}

The proofs of our theorems rely on the vertex decomposition of Theorem~\ref{pontfelbontas}, 
we make use of the quasirandomness of the pairs in the supermatching. However, for proving Theorem~\ref{tetel1} we use different tools from those used for proving Theorem~\ref{tetel2}.

\subsection{The proof of Theorem~\ref{tetel1}}

We need a result, which can be found e.g.~in the survey
paper~\cite{FS} on the dependent random choice method by Fox and Sudakov; the following is Theorem 6.1 in~\cite{FS}.

\begin{theorem}\label{seged1} Let $H$ be a bipartite graph with $h$ vertices and maximum degree $D\ge 1.$ If $d >0$ and
$G$ is a graph with $n \ge 8Dd^{-D} h$ vertices and at least $d \binom{n}{2}$ edges, then $H$ is a subgraph of $G.$
\end{theorem}

\smallskip

With this we are ready for the proof of Theorem~\ref{tetel1}. 

\smallskip

\noindent {\bf Proof:} If $H$ is balanced, let $H'=H.$ If not, then let $H'$ be the disjoint union of two copies of
$H$ arranged so that $H'$ becomes balanced. In either case, we may assume that $H'$ has $2h$ vertices.

If $n$ is odd, leave out an arbitrary vertex of $G.$ For simplicity we keep the notation $n$ for the number of vertices. Then divide $V$ randomly into two parts, $A$ and $B$ such that $|A|=|B|=n/2.$ For that, first distribute the vertices of $G$ into two sets, $A'$ and $B',$ using random coin flipping. By Chernoff's bound
$||A'|-|B'||\le 10\sqrt{n\log n}$ with high probability, moreover, for every $v\in V$ we have that 
$deg(v, A'), deg(v, B') = (1\pm d^4/5)\rho/2$ (here we also use that 
$d\gg 1/n^c$ for any fixed $c>0$). Next, if $|A'|>|B'|,$ then choose $|A'|-n/2$ vertices of $A'$
arbitrarily, and put them into $B'.$ Similarly, if $B'$ is larger, we put $|B'|-n/2$ vertices into $A'.$ Denote the
two parts we obtained by $A$ and $B,$
then  
$deg(v, A), deg(v, B) = (1\pm d^4/4)\rho/2$ for every $v\in V.$

Set $r=(1+d^4/4)\rho/2.$ Simple computation shows that every degree in 
$G[A, B]$ is at least $(1-d^4/2)r,$ and at most $r.$ 
Now we can apply Theorem~\ref{pontfelbontas} to $G,$ with $n$ replaced by $n/2.$
We obtain the $(\ep', d_i)$-bundles $G[A_i, B_i],$ where $\ep'<4\ep^{1/5}.$ 
For every $1\le i\le K$ set $m_i=|A_i|\approx |B_i|.$ Then $m_i\ge d^{101/\ep^2}n/8.$ 

Fix an arbitrary $i\ge 1.$ Assume that we have already embedded vertex disjoint copies of
$H'$ into $G[A_i, B_i],$ and the vacant parts $A'_i\subset A_i$ and $B'_i\subset B_i$ are not
too small: $|A'_i|\ge \ep'|A_i|$ and $|B'_i|\ge \ep'|B_i|.$ Then $e(G[A'_i, B'_i])\ge (d-2\ep')(\ep'm_i)^2$
by $\ep'$-regularity, therefore, $G[A'_i, B'_i]$ has density $>d/3.$ By Theorem~\ref{seged1} we can find a copy of $H'$ in $G[A'_i, B'_i],$ if $$2\ep'm_i \ge 8D(d/3)^{-D}2h.$$ Substituting the bound for $m_i$ we 
conclude that $G[A'_i, B'_i]$ is sufficiently large and dense in order to contain a copy of $H'.$

We repeate this procedure for every $1\le i\le K$ until  the number of vacant vertices drops below
$\ep'm_i.$ Since $H'$ is balanced and the $G[A_i, B_i]$ pairs are almost balanced, when this happens,
altogether less than $3\ep'm_i$ vertices remain vacant. 
Hence, the total number of uncovered vertices in non-exceptional clusters
is at most $3\ep' n/2.$
Since $|A_0|+|B_0|\le 8d^{1/3}n/2$ and $3\ep' \le 12 \ep^{1/5}< d,$ we proved  what was desired. 
\qedf

\subsection{The proof of Theorem~\ref{tetel2}}

The beginning is identical to that of the proof of Theorem~\ref{tetel1}. First divide $V$ into two parts, $A$ and $B$ randomly, as in the proof of Theorem~\ref{tetel1}, then apply Theorem~\ref{pontfelbontas} for $G[A, B].$ 

Next we distribute the vertices of $A_0\cup B_0$ among the non-exceptional clusters $A_i, B_i,$ where $i\ge 1,$
as evenly as possible. We want to assign the vertices of $A_0\cup B_0$ so that no cluster will receive many of them. 
We say that a vertex $v$ 
can be assigned to a non-exceptional cluster $A_i,$ if $deg(v, B_i)\ge \rho|B_i|/(3n);$ similarly, $v$ can be assigned
to a non-exceptional cluster $B_i,$ if $deg(v, A_i)\ge \rho |A_i|/(3n).$ 

Given an arbitrary $v\in A_0\cup B_0,$ there are more than $\rho/3$ vertices in those
non-exceptional clusters to which $v$ can be assigned: this follows from the fact that $v$ can have
less than $|A_0\cup B_0|< 4d^{1/3}n<\rho/3$ neighbors in $A_0\cup B_0,$
and the number of neighbors of $v$ in those clusters to which we cannot assign $v$ is 
at most $\rho/3.$ 

There
are at most $4d^{1/3}n$ vertices to be assigned to non-exceptional clusters, and for each such
vertex we can choose from clusters having a total of at least $\rho/3$ vertices. Hence, even if all
vertices of $A_0\cup B_0$ can be assigned to the same subset of clusters, we can do the assignment
relatively evenly: if a cluster has $m$ vertices, then at most $4d^{1/3}n m/(\rho/3)\le 12d^{2/9}m$
vertices are assigned to it. 

The vertices that we distributed this way are called {\it exceptional } vertices,
and the complement of them is the set of {\it non-exceptional} vertices. 

Our first goal is to incorporate the exceptional vertices into embedded copies of $H.$ Say, that $(S, T)$
is a non-exceptional cluster pair with $|T|=m,$ where $S_0\subset S$ includes the exceptional vertices that were assigned
to the cluster, and we define $T_0\subset T$ analogously. 

Let $S_1\subset S-S_0$ and $T_1\subset T-T_0$ be random subsets, both are obtained by independent coin flippings. Consider the subgraphs $(S_0, T_1)$ and $(S_1, T_0).$ We are going to embed copies of $H$
into these subgraphs until there are only a constant number of vertices remain uncovered in $S_0,$
respectively, $T_0.$ This is done using a simple greedy method, discussed below.

Assume that we have found $k$ vertex disjoint copies of $H$ in $(S_0, T_1),$ and there are still $s$ uncovered vertices in $S_0.$ We make sure that the larger part, $X$ is always mapped to vacant part of $S_0.$ 
Hence, less than $|S_0| \le 12 d^{2/9}m< 25d^{2/9}|T_1|$ vertices are not vacant in $T_1.$ 
In the beginning every $v\in S_0$ had at least $\rho|T_1|/(4n) \ge 2d^{1/9}|T_1|$
neighbors in $T_1$ with high probability using the Chernoff bound. Hence, more than $d^{1/9}|T_1|$ of these are still vacant. Set $\delta=d^{1/9}.$ We need the following.

\begin{claim}\label{moho}
Let $S'\subset S_0$ be an arbitrary subset. Then there exists $u\in T_1$ such that $deg(u, S')\ge \delta |S'|.$
\end{claim}

\noindent {\bf Proof:} Count the edges between $S'$ and $T'_1,$ where $T'_1$ denotes the vacant part of 
$T_1.$ This number is 
at least $|S'|\cdot \delta |T'_1|.$ Taking the average for vertices in $T'_1$ we obtain what was desired. \qedf

\smallskip

We repeatedly apply Claim~\ref{moho}, and get the vacant vertices $u_1, \ldots, u_t \in T_1$ 
such that $deg(u_i, N(u_1, \dots, u_{i-1})\cap S'_0)\ge \delta^i|S'_0|$ for $1\le i\le t,$ where
$S'_0$ denotes the (shrinking) subset of vacant vertices in $S_0.$ 
  
If $t=|Y|$ and $\delta^t|S'_0|\ge |X|,$ then a new copy of $H$ can be embedded. 
It is easy to see that, since $h$ is a constant, we can proceed this way until $|S'_0|$ goes below
a certain constant, which is at most $h/\delta^h.$   

Repeat this procedure for the bipartite subgraph $(S_1, T_0).$ After finishing it, we denote the subset of vacant vertices left in $S$ by $S_2$ and the subset of vacant vertices left in $T$ by $T_2.$ 
Consider now the new pair $(S_2, T_2).$ Since $(S, T)$ was an $(\ep', d)$-super-regular pair, the new pair is $(3\ep', d/3)$-super-regular: we may have $3\ep'$ instead of $\ep'$ since the new cluster sizes
are at least about half of the previous ones, and we added only a constant number of new high degree vertices
(these are the exceptional vertices). Moreover, $S_2$ contains the random subset $S-S_1$ and $T_2$ contains the random subset $T-T_1.$
Hence, by Chernoff's bound every vertex in $S_2\cup T_2$ has sufficiently many neighbors in the opposite cluster. 

Since $H$ is a fixed, unbalanced bipartite graph on $h$ vertices, the size of the larger part divided
by the size of the smaller part is larger than $1+1/h.$ Using this, we may assign (but not map) copies of $H,$ one by one, 
so that vertices of $X$ is assigned to the larger cluster (recall, that $|X|>|Y|$). 
We proceed this way until we
have assigned enough copies of $H$ for balancing out the clusters: set $\Delta=||S_2|-|T_2||$ and $\Delta_H=|X|-|Y|$, then for balancing we need $\lfloor \Delta/\Delta_H\rfloor$ copies of $H.$ 

Next we let $H'$ be the union of two copies $H,$ such that both vertex parts of $H'$ are $X\cup Y.$ 
We continue the assignment with the copies of $H'.$ 
 It is easy to see that after finishing the assignment, the number of vertices assigned to $S_2\cup T_2$ is larger than $|S_2\cup T_2|-h.$ 
 
Finally, we apply the Blow-up lemma~\cite{KSSzBl, KSSzBl2} for embedding the vertex disjoint copies of
$H$ we assigned to the pair $(S_2, T_2).$ 

Repeate this procedure for every 
$(A_i, B_i)$ ($i\ge 1$) pair in the decomposition. Using the previous observation, 
at most $C\le K\cdot h\le \frac{8h}{d}\cdot d^{-100/\varepsilon^2}$ vertices remain uncovered.  
\qedf


\begin{thebibliography}{99}

\bibitem{AY} N.~Alon, R.~Yuster (1996), {\it $H$-factors in dense graphs.} J.~Combin.~Theory B {\bf 66} 269--282.  

\bibitem{CF} D.~Conlon and J.~Fox, Bounds for graph regularity and removal lemmas, {\it GAFA} {\bf 22} (2012)
1191--1256.

\bibitem{CH} K.~Corr\'adi, A.~Hajnal, (1963), On the maximal number of independent circuits in a graph. ´
{\it Acta Math. Acad. Sci. Hungar.} {\bf 14} 423-439.

\bibitem{Cs} B.~Csaba, {\it A stability theorem for embedding bounded degree spanning trees,} manuscript.

\bibitem{aus} B.~Csaba, {\it Regular decomposition of the edge set of a graph with applications,} Australasian Journal of Combinatorics, {\bf 89(2)} (2024), 249-267.


\bibitem{FS} J.~Fox, B.~Sudakov (2011), {\it Dependent random choice.} Random Structures \& Algorithms. {38\bf } (1-2): 
68--99.  doi:10.1002/rsa.20344

\bibitem{Gowers} W.~T.~Gowers,  {\it Lower bounds of tower type for Szemer\'edi’s uniformity lemma,} {\it GAFA} {\bf 7} (1997)
322-337. 


\bibitem{GRR} R.~Graham, V.~R\"odl and A.~Ruci\'nski, {\it On bipartite graphs with linear Ramsey numbers,} Combinatorica {\bf 21} (2) (2001) 199--209.

\bibitem {HSz} A.~Hajnal, E.~Szemer\'edi, (1970) Proof of a conjecture of Erd\H os. In {\it Combinatorial Theory
and its Applications,} Vol. II (P. Erd\H os, A. R\'enyi and V. T. S\'os, eds.), Colloq. Math. Soc. J. Bolyai {\bf 4}, North-Holland, Amsterdam, pp. 601-623.

\bibitem{JLR} S.~Janson, T.~Luczak, A.~Ruczinski, Random graphs, volume 45. John Wiley \& Sons, 2011.

\bibitem{KohRdl} Y.~Kohayakawa, V.~R\"odl, {\it Szemer\'edi’s regularity lemma and quasi-randomness.} In: Recent advances in algorithms and combinatorics. New York, NY: Springer New York, 289-351, 2003.

\bibitem{KSSzBl} Koml\'os, J., S\'ark\"ozy, G. N., Szemer\'edi, E. {\it Blow-up lemma} Combinatorica, 17(1):109 - 123, 1997.

\bibitem{KSSzBl2} Koml\'os, J., S\'ark\"ozy, G. N., Szemer\'edi, E. {\it An Algorithmic Version of the Blow-up
Lemma,} Random Struct. Alg., 12, 297-312, 1998.

\bibitem{KSSz} J.~Koml\'os, G.~S\'ark\"ozy, E.~Szemerédi (2001) {\it Proof of the Alon–Yuster conjecture.} ´
Discrete Math. {\bf 235} 255--269.


\bibitem{KS} J.~Koml\'os, M.~Simonovits (1996), {\it Szemer\'edi’s Regularity Lemma and its applications in graph
theory.} In: Combinatorics, Paul Erd\H os is Eighty, Vol II (D. Mikl\'os, V. T. S\'os, T. Sz\H onyi eds.),
J\'anos Bolyai Math. Soc., Budapest 295-352.

\bibitem{KO} D.~K\"uhn, D.~Osthus (2005) {\it Packings in dense regular graphs.} Combinatorics, Probability and Computing, {\bf 14(3)}: 325-337. 

\bibitem{KO09} D.~K\"uhn, D.~Osthus (2009) {\it The minimum degree threshold for perfect graph packings.} Combinatorica, {\bf 29(1)}: 65-107.

\bibitem{LMS} S.~Letzter, A.~Methuku, B.~Sudakov (2026) {\it Packing subgraphs in regular graphs}, \url{https://doi.org/10.48550/arXiv.2509.26180}

\bibitem{MPRT} R.~Montgomery, K.~Petrova, A.~Ranganathan, J.~Tan (2026) {\it Packing subdivisions into regular graphs}, \url{https://doi.org/10.48550/arXiv.2508.00480}


\bibitem{Peng} Y.~Peng, V.~Rödl, A.~Ruci\'nski, {\it Holes in graphs,} The Electronic Journal of Combinatorics {\bf 9} (2002), \#R1.



\bibitem{Szemeredi} E.~Szemerédi (1976), {\it Regular partitions of graphs,} Colloques Internationaux C.N.R.S. No 260
- Probl\'emes Combinatoires et Théorie des Graphes, Orsay 399--401.



\bibitem{V} J.~Verstra\"ete, (2002) {\it A note on vertex-disjoint cycles.} Combin. Probab. Comput. {\bf 11 }
97--102.


\end{thebibliography}
\end{document}